\documentclass{article}
\usepackage{graphicx} 
\usepackage[dvipsnames]{xcolor}
\usepackage{setspace}
\usepackage[hang,flushmargin,symbol*]{footmisc}
\usepackage{amsmath}
\usepackage{amsthm}
\usepackage{amssymb}
\usepackage{mathtools}
\usepackage{enumitem}
\usepackage{calc}
\usepackage{graphicx}
\usepackage{caption}
\usepackage[labelformat=simple,labelfont={}]{subcaption}
\usepackage{tikz}
\usepackage{tikz-cd}
\usetikzlibrary{decorations.markings}
\usetikzlibrary{arrows,shapes,positioning}
\usepackage{url}
\usepackage{array}
\usepackage{graphicx}
\usepackage{color}
\usepackage{mathrsfs}

\usepackage{verbatim}

\usepackage{subfiles}

\usepackage[colorinlistoftodos, loadshadowlibrary]{todonotes}

\usepackage{dashbox}

\theoremstyle{definition}
\newtheorem{theorem}{Theorem}[section]
\newtheorem{lemma}[theorem]{Lemma}

\newtheorem{corollary}[theorem]{Corollary}

\newtheorem{definition}[theorem]{Definition}
\newtheorem{example}[theorem]{Example}
\newtheorem{remark}[theorem]{Remark}

\usepackage{marginnote}

\usepackage{amsfonts}
\usepackage{amsmath}
\usepackage{amssymb}
\usepackage{mathrsfs}
\usepackage{tikz}
\usepackage{tikz-cd}
\usepackage[colorlinks=true]{hyperref}
\usepackage{ mathdots }
\usepackage{dutchcal}
\usepackage{verbatim}
\usepackage{amsmath,amsthm}
\usepackage{extarrows}

\usepackage{xparse}

\usepackage{multirow}

\usepackage[normalem]{ulem}

\renewcommand{\tilde}[1]{\ensuremath{\widetilde{#1}}}

\newcommand{\msout}[1]{\text{\sout{\ensuremath{#1}}}}

\newcommand{\OO}{\mathcal{O}}

\newcommand{\Aff}{{\mathbb{A}}}

\newcommand{\GG}{\mathbb{G}}

\newcommand{\Spec}{{\text{Spec}\:}}

\newcommand{\Log}{{\mathcal{L}}}

\newcommand{\lpbstrict}{{\arrow[dr, phantom, very near start, "\ulcorner \msout{\ell}"]}}

\newcommand{\action}{\:\rotatebox[origin=c]{-90}{$\circlearrowright$}\:}

\newcommand{\num}[1]{{\langle #1 \rangle}}

\newcommand{\Sym}{{\rm Sym}}

\renewcommand{\bar}[1]{\overline{#1}}

\newcommand{\HHl}[1]{{\rm HH}^\ell_{#1}}

\newcommand{\D}{\mathbf{D}}
\newcommand{\Aut}{\underline{\text{Aut}}}

\usepackage{stmaryrd}

\renewcommand{\i}{\tilde{i}}

\definecolor{sebgreen1}{rgb}{0.019,0.317,0.149}
\definecolor{sebgreen2}{rgb}{0.784,0.952,0.780}

\newcommand{\af}[1]{\Theta_{#1}}

\newcommand{\scr}[1]{{\ensuremath{\mathscr{#1}}}}

\newcommand{\bra}[1]{{\left[{#1}\right]}}

\title{Logarithmic Hochschild (co)homology of logarithmic orbifolds}

\author{Márton Hablicsek, Leo Herr, Francesca Leonardi}

\date{\today}

\begin{document}

\maketitle

\begin{abstract}
    Recently, the authors of this paper introduced logarithmic Hochschild (co)homology of logarithmic spaces in a geometric way using formality of derived intersections. In this paper, the authors extend the decomposition theorem for the logarithmic Hochschild (co)homology of firm orbifolds to general logarithmic orbifolds and consider two applications of the decomposition theorem. First, we consider two versions of a symmetric product and compute the logarithmic Hochschild homology of them. Second, we show that logarithmic Hochschild homology is invariant under root stack operations. 
\end{abstract}

\section{Introduction}
Let $X$ be a quasicompact, quasiseparated, weakly log separated\footnote{Recall that ``weakly log separated'' means the morphism $X \to \af{X}$ to the Artin fan is separated.}, finite type log algebraic stack over a field of characteristic 0. Recently, in \cite{HABLICSEK2026127}, the logarithmic Hochschild (co)homology of $X$ was defined via derived intersections. Namely, the authors defined the \textit{log Hochschild homology} $\HHl{X}$ as the endofunctor of a dg-enhancement of the derived category of coherent sheaves on $X$
\[i^*i_*:\D(X)\to \D(X)\]
where $i$ is the \textit{log diagonal} map $i:X\to X\times_{\af{X}}X$. Here the functors are derived functors and $\af{X}$ denotes the Artin fan of the log stack $X$. Throughout the paper, all functors will be derived unless stated otherwise.

The $n$-th Hochschild homology group of $X$ is defined as the hypercohomology groups of the endofunctor evaluated at the structure sheaf
\[R^n\Gamma(X,\HHl{X}(\OO_X)).\]
The authors proved that this definition of log Hochschild homology agrees with the definition of Olsson \cite{olsson2024loghochschild}.

The logarithmic Hochschild cohomology of $X$ is defined similarly via the endofunctor $i^!i_*:\D(X)\to \D(X)$.

As one of the applications, the authors showed an ``orbifold HKR'' decomposition result \cite[Theorem D]{HABLICSEK2026127} for the Hochschild (co)homology for certain ``firm'' logarithmic orbifolds $[X/G]$ where the group action is trivial on the Artin fan. 
The goal of this short paper is to extend the results of \cite{HABLICSEK2026127} in the orbifold direction. To do so, we prove the decomposition result in the case of general log orbifolds and drop the firm action assumption.

\begin{theorem}\label{thm:main}
Let $X$ be a log smooth scheme, $G$ a finite group acting on $X$. Then we have isomorphisms of $k$-vector spaces
\[R^n\Gamma(\HHl{[X/G]}(\OO_{[X/G]}))=\left(\bigoplus_{g\in G} R^n\Gamma(\HHl{X^g_{\log}}(\OO_{X^g_{\log}}))\right)^G=\]
\[=\left(\bigoplus_{g\in G} \bigoplus_{q-p=n}H^p(X^g_{\log}, \Omega^{q,\log}_{X^g_{\log}})\right)^G.\]
\end{theorem}

We have two applications in mind for Theorem \ref{thm:main}. First, we consider two version of a symmetric product and compute the logarithmic Hochschild (co)homology of the two versions. Second, we generalize Theoreom C of \cite{HABLICSEK2026127} and show that if $(X,D)$ is a smooth pair ($X$ is a smooth variety, $D$ is a smooth Cartier divisor), then the logarithmic Hochschild (co)homology of the $n$-th root stack $\sqrt[n]{(X,D)}$ agrees with the logarithmic Hochschild (co)homology of $(X,D)$. This shows that logarithmic Hochschild (co)homology is invariant under log alteration (under blow-ups and root stacks). Strictly speaking, we do not need Theorem \ref{thm:main} for the root stack case, our previous theorem (Theorem D of \cite{HABLICSEK2026127}) suffices (even though, in general, the root stack is not a global quotient orbifold).

\textbf{Acknowledgement:} M.H. was supported by the MTA Distinguished Guest Scientist Fellowship Programme 2026 and some part of the work was done at the Budapest University of Technology and Economics. M.H. thanks \'Ad\'am Gyenge for important conversations.

\section{Log Hochschild of Orbifolds}

Let $(X, M_X)$ be a log scheme equipped with a log structure $M_X$, and assume that a finite group $G$ acts on $X$ that respects the log structure. In other words, $G$ acts on the sheaf of monoids $M_X$ compatible with the action on the structure sheaf $\OO_X$. 

The orbifold $[X/G]$ has a natural log structure. Explicitly, the sheaf of monoids $M_{[X/G]}$ on $[X/G]$ is given as follows. A test scheme $T\to [X/G]$ is given as a $G$-torsor $P\to T$ and a $G$-equivariant map $f:P\to X$. 
The sections of the monoid $M_{[X/G]}(T)$ are defined as the $G$-invariant sections of $(f^*M_X)(P)$, i.e $\left((f^*M_X)(P)\right)^G$. 

\begin{lemma}
    The natural quotient map $\pi:X\to [X/G]$ is strict.
\end{lemma}

\begin{proof}
    We compare the sheaf of monoids $\pi^*M_{[X/G]}$ and $M_X$. Note that the map $\pi: X\to [X/G]$ corresponds to the trivial torsor $G\times X\to X$ with the $G$-equivariant action map $G\times X\to X$. Now, consider a test scheme $S\to X$.
    
    We will show that $\pi^*M_{[X/G]}(S)=M_X(S)$. Using the quotient map $\pi:X\to [X/G]$, we can regard $S$ as a scheme over $[X/G]$ and equate $S\times_{[X/G]}X=G\times S$. Thus, $\pi^*M_{[X/G]}(S)=(M_X(G\times S))^G=M_X(S)$ proving our claim. 
\end{proof}

Using the universal property of the Artin fan $\af{X}$, we get a natural map $\af{X}\to \af{[X/G]}$ making the following diagram commutative
\[
\begin{tikzcd}
    X \ar[r, "\pi"]\ar[d] & {[X/G]} \ar[d]\\
    \af{X}\ar[r] & \af{[X/G]}.
\end{tikzcd}
\]
There's an induced action $G \action \af{X}$ on the Artin fan, and $\af{\bra{X/G}}$ is the relative coarse moduli space of $\bra{\af{X}/G}$ over $\Log$.


To compute the logarithmic Hochschild co/homology of the orbifold $[X/G]$, we need to consider the self-intersection of the log diagonal
\[ \iota: [X/G] \to [X /G] \times_{\af{[X/G]}} [X/G]. \]
Since the Artin fan $\af{[X/G]}$ has a trivial $G$-action, the latter space is isomorphic to $[X\times_{\af{[X/G]}} X / G \times G]$ where $G\times G$ acts componentwise. Therefore, the endofunctors $\iota^* \iota_*$ and $\iota^! \iota_*$ decompose as
\[\iota^* \iota_* \simeq \bigoplus_{g \in G} \i_{g}^* \i_* \quad \mathrm{and} \quad \iota^! \iota_* \simeq \bigoplus_{g \in G} \i_{g}^! \i_*\]
where $\i, \i_g : X \to X \times_{\af{[X/G]}} X$ are defined as the diagonal $\i(x)=(x,x)$ and twisted diagonal maps $\i_g(x)=(x,g.x)$ as in \cite{arinkin2019orbifold}. 

We analyze the space $X\times_{\af{[X/G]}}X$ and the maps $\i$, $\i_g$.

\begin{lemma}
    For a logarithmic scheme $X$, the log algebraic stack $X \times_{\af{[X/G]}} X$ is representable by log algebraic spaces.
\end{lemma}

\begin{proof}
    The log algebraic stack is defined by the pullback diagram
    \[
    \begin{tikzcd}
        X \times_{\af{[X/G]}} X \ar[d] \ar[r] \lpbstrict & X \times X \ar[d] \\
        \af{[X / G]} \ar[r] & \af{[X / G]} \times \af{[X/G]}
    \end{tikzcd}.
    \]
    The bottom horizontal map is representable by log algebraic spaces having $\af{[X/G]} \times \af{[X/G]} \simeq \af{[X/G] \times [X/G]}$ (see \cite[Proposition 3.3]{HABLICSEK2026127}) and thus being a map of Artin fans. So, the top horizontal map is representable by log algebraic spaces, and its target is a log scheme, therefore, its source is representable by algebraic spaces.
\end{proof}

\begin{lemma}
    The map $X \times_{\af{[X/G]}} X \to X \times X$ is log étale.
\end{lemma}

\begin{proof}
    The map $X \times_{\af{[X/G]}} X \to X \times X$ is obtained by the strict pullback diagram
    \[
    \begin{tikzcd}
        X\times_{\af{[X/G]}} X \ar[d] \ar[r] \lpbstrict & X \times X \ar[d] \\
        \af{[X/G]} \ar[r] & \af{[X/G]} \times \af{[X/G]}
    \end{tikzcd}.
    \]
    Since $\af{[X/G]} \times \af{[X/G]} \simeq \af{[X/G] \times [X/G]} \simeq \af{[X \times X / G \times G]}$ (\cite{HABLICSEK2026127}) the bottom horizontal map is log étale being a map of Artin fans. Thus, the top horizontal map is a log \'etale map as well, because log étale maps are stable under log pullbacks.
\end{proof}

\begin{lemma}
    The maps $\i, \i_g : X \to X\times_{\af{[X/G]}} X$ are strict.
\end{lemma}

\begin{proof}
    The maps $\i, \i_{g}$ are defined by the diagram
    \[
    \begin{tikzcd}
        X \ar[r, "\i\mathrm{,}\i_g"'] \ar[dr] \ar[rr, bend left = 20, "\Delta \mathrm{,} \Delta_g"] & X \times_{\af{[X/G]}} X \ar[d] \ar[r] \lpbstrict & X \times X \ar[d] \\
        & \af{[X/G]} \ar[r] & \af{[X/G]} \times \af{[X/G]}
    \end{tikzcd}.
    \]
    The right-most vertical map is strict being given by two copies of the composite of strict maps $X \to \af{X} \to \af{[X/G]}$. Hence, its pullback $X \times_{[X/G]} X \to \af{[X/G]}$ is also strict. As before, the map $X \to \af{[X/G]}$ is strict. Therefore, $\i,\i_g$ are as well.
\end{proof}

\begin{lemma}
    If $X$ is log smooth, then so is $X \times_{\af{[X/G]}} X$.
\end{lemma}

\begin{proof}
    If $X$ is log smooth, then so is $X\times_{\Log}X$. The smoothness of $X \times_{\af{[X/G]}} X$ follows because the natural map $X\times_{\af{[X/G]}}X\to X\times_{\Log}X$ induced by $\af{[X/G]}\to \Log$ is an open immersion, hence smooth (\cite{HABLICSEK2026127}).
\end{proof}

\begin{lemma}
    If $X$ is log smooth and log separated, then the diagonal and twisted diagonal maps $\i, \i_g : X \to X \times_{\af{[X/G]}} X$ are l.c.i. closed immersions.
\end{lemma}

\begin{proof}
    The maps $\i_g$ induce distinguished triangles
    \[\i_g^* \mathbb L^{\log}_{X\times_{\af{[X / G]}} X} \to \mathbb L^{\log}_{X} \to \mathbb L^{\log}_{\i_g} \xrightarrow{+1}.\]
    As the maps $\i_g$ are strict, $\mathbb L ^{\log}_{\i_g} \simeq \mathbb L_{\i_g}$. Since the source and target are log smooth, their cotangent complex is concentrated at most in degree 0. Furthermore, the maps $\i_g:X\to X\times_{\af{[X/G]}}X$ split via the projection to the first component $X\times_{\af{[X/G]}}X\to X$, implying that $\mathbb L_{\i_g}$ is concentrated in degree $-1$. Thus, the maps $\i_g$ are l.c.i.\ morphisms.
    
    To show that the maps are closed immersions when $X$ is log separated we look at the diagram
    \[
    \begin{tikzcd}
        X \ar[r] & X \times_{\af{[X / G]}} X \ar[r] \ar[d] \lpbstrict & X \times_\Log X \ar[d] \\
        & \af{[X / G]} \ar[r] &\af{[X / G]} \times_\Log \af{[X / G]}.
    \end{tikzcd}
    \]
    The composite of the two top horizontal maps $X\to X\times_{\Log}X$ is proper by definition since $X$ is log separated. Since $\af{[X / G]} \to \Log$ is strict, étale and representable by algebraic spaces, its diagonal is an open immersion. Thus, also $\i$ and $\i_g$ are proper maps implying that the maps $\i_g$ are l.c.i closed immersions.
\end{proof}

\begin{remark}
    When $X$ is not log separated, the diagonal and twisted diagonal maps are locally l.c.i closed immersions meaning that they are l.c.i closed immersions into an open subset of $X\times_{\af{[X/G]}}X$.
\end{remark}

Now, we turn our attention to understanding the intersections of the twisted diagonal maps $\i_g$.

\begin{definition}\label{def:fixedlocus}
    We define the \textit{log fixed locus}, $X_{\log}^g$, as the (underived) intersection of the diagonal map $\i$ and the twisted diagonal map $\i_g$.
    \[
    \begin{tikzcd}
        X^g_{\log}\ar[r, "p"]\ar[d, "q"] & X\ar[d, "\i"]\\ X\ar[r, "\i_g"] & X \times_{\af{[X/G]}} X
    \end{tikzcd}
    \]
\end{definition}

Definition \ref{def:fixedlocus} is the extension of the definition used in \cite{HABLICSEK2026127} for the log fixed point locus of a firm action. In fact, an action is defined to be \textit{firm} if the action of $G$  on the Artin fan $\af{X}$ is trivial. In that case, $\af{[X/G]}=\af{X}$, and the definition in \cite{HABLICSEK2026127} agrees with the definition above.

\begin{remark}
    We remark that the definition above does not depend on the Artin fan we chose. Let us consider an Artin fan $\scr B$ (with a trivial $G$-action) given by a strict, necessarily \'etale map $\af{[X/G]}\to\scr B$, and consider the product $X\times_{\scr B}X$ and the corresponding twisted diagonal maps $j_g:X\to X\times_{\scr B}X$. We could define the log fixed locus $X^g_{\log}$ as the intersection of $j:X\to X\times_{\scr B}X$ and the twisted diagonal $j_g:X\to X\times_{\scr B}X$. Since the map 
    $X\times_{\af{[X/G]}}X\to X\times_{\scr B}X$ induced by the strict \'etale map $\af{[X/G]}\to \scr B$ is an open immersion, we see that the log fixed locus remains unchanged. 
\end{remark}

\begin{definition}
    Given the above intersection diagram, we denote by $E$ the associated excess bundle of $q$ defined by the exact sequence
    \begin{equation}\label{eqn:excessbundledef}
        0 \to E^\vee \to N^\vee_{\tilde i}|_{X_{\log}^g} \to N^\vee_{q} \to 0.
    \end{equation}
\end{definition}

We have the following statement about the excess bundle.

\begin{lemma}\label{lem:logfixedlocussmooth}
    The log fixed point locus is smooth. Furthermore, we have isomorphisms $E^\vee\cong (\Omega^{1, \log}_X)_g|_{X_{\log}^g}\cong\Omega^{1, \log}_{X^g_{\log}}$.
\end{lemma}

\begin{proof}
    Consider the exact sequence \eqref{eqn:excessbundledef}. Since the map $\i: X\to X\times_{\af{[X/G]}}X$ is strict, we have that 
    \[N^\vee_{X / X \times_{\af{[X/G]}} X}=N^{\log, \vee}_{X/X\times_{\af{[X/G]}}X}.\]
    Since the map $X\times_{\af{X}}X\to X\times_{\af{[X/G]}}X$ is an open immersion, we get that 
    \[N^{\log, \vee}_{X/X\times_{\af{[X/G]}}X}=N^{\log, \vee}_{X/X\times_{\af{X}}X}\]
    This latter sheaf can be identified with $\Omega^{1,\log}_X$ by \cite[Lemma 4.5]{HABLICSEK2026127}.
    
    As $X^g_{\log}$ is the intersection of the diagonal and twisted diagonal maps, $N^{\vee}_{X^g_{\log}/X}$ is given by the cokernel of $id-g^*:\Omega^{1, \log}_{X}\to \Omega^{1, \log}_{X}$ on $X^g_{\log}$. The kernel of this map is precisely $(\Omega^{1,\log}_X)_g$, so we obtain a short exact sequence
    \[0\to (\Omega^{1,\log}_X)_g\to \Omega^{1,\log}_X|_{X^g_{\log}}\to N^{\vee}_{X^g_{\log}/X}\to 0.\]
    
     This sequence can be identified with the excess short exact sequence and the logarithmic conormal short exact sequence. This shows that $E^\vee\cong (\Omega^{1, \log}_X)_g|_{X_{\log}^g}\cong\Omega^{1,\log}_{X^g_{\log}}$. From this, we get that $(T_X^{\log})^g\cong T^{\log}_{X^{g}_{\log}}$ implying that $X^g_{\log}$ is log smooth.
\end{proof}

We can now compute the endofunctors $\i_{g}^* \i_*$ and $\i_{g}^! \i_*$ using \cite[Theorem 5.11]{HABLICSEK2026127} (which is a generalization of  \cite[Proposition 3.3]{arinkin2019orbifold}, or see \cite{grivaux2014hochschild} ). There are three assumptions that are needed to be satisfied: a) the excess bundle is locally free, b) the map $X\to X\times_{\af{[X/G]}}X$ splits to first order, c) the short exact sequence
\[0 \to E^\vee \to N^\vee_{X / X \times_{\af{[X/G]}} X}|_{X_{\log}^g} \to N^\vee_{X_{\log}^g/X} \to 0\]
splits. 

In our case, all 3 assumptions are satisfied: a) and c) from Lemma \ref{lem:logfixedlocussmooth} and $X\to X\times_{\af{[X/G]}}X$ is split via any of the two projections $X\times_{\af{[X/G]}}X\to X\times X\to X$. As a result, we obtain an isomorphism of dg endofunctors $\D(X)\to \D(X)$:
\[\i_g^* \i_* (-) \simeq q_* (p^* (-) \otimes \Sym(\Omega^1_{X^g_{\log}} [1])) \qquad \i_g^! \i_* (-) \simeq q_* (p^! (-) \otimes \Sym(T_{X^g_{\log}} [-1])).\]

As a corollary, we obtain Theorem \ref{thm:main}. The proof of this theorem is completely parallel to \cite{arinkin2019orbifold} and \cite{HABLICSEK2026127}.

\begin{corollary}[Orbifold HKR]
    We have isomorphisms of $k$-vector spaces
    \[R^n\Gamma(\HHl{[X/G]}(\OO_{[X/G]}))=\left(\bigoplus_{g\in G} R^n\Gamma(\HHl{X^g_{\log}}(\OO_{X^g_{\log}}))\right)^G=\]
    \[=\left(\bigoplus_{g\in G} \bigoplus_{q-p=n}H^p(X^g_{\log}, \Omega^{q,\log}_{X^g_{\log}})\right)^G.\]
\end{corollary}

\section{Applications}

In this section, we apply the decomposition theorem on log Hochschild homology in two cases: a) symmetric products and b) root stacks.

\subsection{Symmetric product}

Let $X$ be a log smooth variety (for instance, we consider a smooth variety $X$ with log structure coming from a simple normal crossing divisor). We can consider two symmetric products: the symmetric product $[X^n/S_n]$, i.e the permutation action of $S_n$ on $X^n$, or the symmetric product $[\widehat{X^n}/S_n]$ where $\widehat{X^n}=X\times_{\af{X}}X\times_{\af{X}}...\times_{\af{X}}X$ is the $n$-th power of $X$ over its Artin fan. We see that, indeed, $S_n$ respects the natural log structure on $X^n$ and $\widehat{X^n}$ and that both $X^n$ and $\widehat{X^n}$ are log smooth. 

We begin by analyzing the easier case, namely the symmetric product $[\widehat{X^n}/S_n]$.

\subsubsection{Symmetric product $[\widehat{X^n}/S_n]$}

\begin{lemma}
    The Artin fan of $\widehat{X^n}$ is $\af{X}$. The action of $S_n$ is firm, and thus the Artin fan of $[\widehat{X^n}/S_n]$ agrees with $\af{X}$ as well. 
\end{lemma}

\begin{proof}
    Note that $\widehat{X^n}$ is obtained via Cartesian products of the log scheme $X$ over $\af{X}$. Therefore its Artin fan is $\af{X}$ (\cite{HABLICSEK2026127}). The induced $S_n$-action is trivial on the Artin fan, and thus the action is firm, showing that $\af{\widehat{X^n}}=\af{X}$ also (\cite{HABLICSEK2026127}).
\end{proof}

To compute the log Hochschild homology of $\widehat{X^n}$, we need to compute the log fixed loci i.e the intersection of the log diagonal and the twisted diagonals $i_\sigma:\widehat{X^n}\to \widehat{X^n}\times_{\af{X}}\widehat{X^n}$. 

\begin{lemma}
    Consider the log fixed point locus
    \[
    \begin{tikzcd}
        (\widehat{X^n})_{\log}^\sigma \ar[r] \ar[d] \lpbstrict & \widehat{X^n} \ar[d] \\
        \widehat{X^n} \ar[r] & \widehat{X^n} \times_{\af{X}} \widehat{X^n}.
    \end{tikzcd}
    \]
    Then $(\widehat{X^n})_{\log}^\sigma \simeq \widehat{X^{\# O(\sigma)}}$ where $O(\sigma)$ is the set of disjoint orbits of $\sigma$ acting on $\{1, \dots, n\}$.
\end{lemma}

\begin{proof}
    Note that it is enough to show the statement locally. Therefore, we assume that $V$ is a toric variety with dense torus $T$ with its natural log structure. In this case, $\af{V}=[V/T]$, and the symmetric product $\widehat{V^n}\times_{\af{V}}\widehat{V^n}$ can be identified with $V\times T^{n-1}\times (T^{n-1}\times T)$. The identification
    \[V\times T^{n-1}\times (T^{n-1}\times T)\to \widehat{V^n}\times_{\af{V}}\widehat{V^n}\] is given by
    \[(v,t_1,...,t_{n-1}), (t_1',.., t_{n-1}',t)\mapsto ((v, vt_1,..., vt_{n-1}), (tv, tvt_1', ..., tvt_{n-1}')).\]
    With this identification, the diagonal map is given by 
    \[(v,t_1,...,t_{n-1})\mapsto ((v,t_1,..., t_{n-1}), (t_1,...,t_{n-1}, 1)).\]
    Indeed, on $\widehat{V^n}\times_{\af{V}}\widehat{V^n}$, the diagonal map is given by 
    \[(v,t_1,...,t_{n-1})\mapsto ((v,t_1,..., t_{n-1}),(v,t_1,..,t_{n-1}))\] which translates to the above via the identification
    \[V\times T^{n-1}\times (T^{n-1}\times T)\to \widehat{V^n}\times_{\af{V}}\widehat{V^n}.\]
    Similarly, the twisted diagonal map corresponding to a permutation $\sigma\in S_n$ can be described by sending the point $(v,t_1,..., t_{n-1})$ to
    \[(v,t_1,..., t_{n-1}), (t_{\sigma^{-1}(0)}^{-1}t_{\sigma^{-1}(1)}, t_{\sigma^{-1}(0)}^{-1}t_{\sigma^{-1}(2)},...,t_{\sigma^{-1}(0)}^{-1}t_{\sigma^{-1}(n-1)},t_{\sigma^{-1}(0)})\]
    (where we regard $\sigma$ as a permutation on the letters $0,1,.., n-1$ and set $t_0$ to be 1).
    
    With some simple combinatorics, we get that the log fixed locus is given by the orbits of the action, namely, the log fixed locus can be identified with $V\times T^{\#O(\sigma)-1}=\widehat{X^{\#O(\sigma)}}$.
\end{proof}

\begin{example}
    Consider the case of $\widehat{X^2}$ with its $S_2$ action. In this case, the diagonal map (locally) can be described as $(v,t_1)\mapsto (v,t_1, t_1, 1)$ while the twisted diagonal for the non-trivial element is $(v,t_1)\mapsto (v,t_1, t_1^{-1}, t_1)$. Therefore, the log fixed locus can be described as the points $(v, 1)$, i.e with $X$.
\end{example}

\begin{remark}\label{rem:blowup}
    This example shows that the log fixed locus and the usual fixed locus do not agree in general. In fact, in the case of a log smooth pair $(X,D)$ (with $X$ being smooth and $D$ being a smooth divisor), $\widehat{X^2}$ can be described as the open set of the blow-up $B$ of $X\times X$ along $D\times D$ that is the complement of the strict transforms of the divisors $D\times X$ and $X\times D$. The usual fixed locus is the disjoint union of the diagonal $X\to B$ and a copy of $D$ on the exceptional divisor, while the log fixed locus is only $X\to B$.
\end{remark}

Using the HKR decomposition for orbifolds, we get the following result, similar to \cite{belmans2023hochschildcohomologyhilbertschemes}.

\begin{corollary}
    We have an isomorphism of $k$-vector spaces
    \[R^n\Gamma(\HHl{[\widehat{X^n}/S_n]}(\OO_{[\widehat{X^n}/S_n]})=\left(\bigoplus_{\sigma\in S_n}R^n\Gamma(\HHl{\widehat{X^{\#O(\sigma)}}}(\OO_{\widehat{X^{\#O(\sigma)}}})\right)^{S_n}.\]
\end{corollary}

While the Corollary above in general gives us infinite dimensional vector spaces, following Remark \ref{rem:blowup}, in the case of $n=2$, and $X$ being a log smooth pair, we can interpret it for the blow-up $B$. For the orbifold $[B/S_2]$, the usual (non-log) Hochschild homology decomposes as
\[HH_*([B/S_2])=HH_*(B)^{S_2}\oplus HH_*(B^{S_2})=\]
\[=\Sym^2HH_*(X)\oplus \Sym^2HH_*(D)\oplus HH_*(X)\oplus HH_*(D),\]
where the first equality comes from \cite{arinkin2019orbifold} and the second equality comes from \cite{anno2025hochschild, nordstrom2025decomposition}.

On the other hand, the log Hochschild homology decomposes as 
\[R^*\Gamma \HHl{[B/S_2]}(\OO_{[B/S_2]})=R^*\Gamma\HHl{B}(\OO_B)^{S_2}\oplus R^*\Gamma\HHl{X}(\OO_X)=\]
\[=\Sym^2R^*\Gamma \HHl{X}(\OO_X)\oplus R^*\Gamma\HHl{X}(\OO_X).\]
Here, the first equality comes from Theorem \ref{thm:main}, and the second equality comes from the fact that $B\to X\times X$ is log \'etale and log Hochschild homology is invariant under log alterations \cite{HABLICSEK2026127}.

\begin{remark}
    The above computation suggests that from a Heiseinberg representation point of view \cite{gyenge2025heisenberg, anno2025hochschild, nordstrom2025decomposition}, a right representative for the second symmetric product of $(X,D)$ is the log orbifold $[B/S_2]$. 
\end{remark}

\subsubsection{Symmetric product $[X^n/S_n]$}

Although the computation for the symmetric product $[\widehat{X^n}/S_n]$ resembles the usual computation of symmetric products for the Hochschild homology, the story for $[X^n/S_n]$ is quite different.

We start with the following observation.

\begin{example}\label{ex:counterexamplelogfixedlocus}
    Let $X$ be a separably closed field with rank-one log structure. The map $X^2 \to \Log$ factors through the inclusion of a stacky point $[\Sym^2 {\GG_m}] = B \mu_2 \ltimes \GG_m^2$. Writing $G = \mu_2 \ltimes \GG_m^2$ for this affine group scheme, the fiber product $X^2 \times_\Log X^2$ can be identified with $X^2 \times G$, where the projection onto the second factor is the ``difference'' between the two points over $\Log$. 

    The map $i : X^2 \to X^2 \times G$ is the inclusion $(x, y) \mapsto ((x, y), 0)$. The twisted diagonal $i_{(12)}$ sends $(x, y)$ to $((x, y), (-1, 1))$, where by $(-1, 1)$ we mean the nontrivial element of $\mu_2$ and the trivial element of $\GG_m$. The intersection of $i$ and $i_{(12)}$ inside $X \times_{\Log} X = X \times BG$ is therefore the empty set. So the log fixed locus is empty. 
\end{example}

After this illustrative example, we analyze the log fixed loci of the symmetric product $[X^n/S_n]$.

First, note that we have two projection maps $X^n\times_{\af{[X^n/S_n]}}X^n\to X^n\times X^n\times X^n$. This shows that we have a natural map betwen the log fixed locus and the usual (non-log) fixed locus $(X^n)^\sigma_{\log}\to (X^n)^\sigma=X^{\# O(\sigma)}$. 

\begin{lemma}
    Let $\sigma \in S_n$ be a transitive permutation with $n \geq 2$. Then $(X^n)^\sigma_{\rm log}$ is the usual fixed point subscheme $(X_0^n)^\sigma = X_0$ of the locus where the log structure is trivial. 

    For an arbitrary permutation $\sigma \in S_n$, write $F(\sigma) \subseteq O(\sigma) = \num{n}/\sigma$ for the fixed points corresponding to singleton orbits and $S(\sigma) = O(\sigma) \setminus F(\sigma)$. The map $(X^n)^\sigma_{\rm log} \to (X^n)^\sigma = X^{O(\sigma)}$ is the product $(X_0)^{\#S(\sigma)} \times X^{\#F(\sigma)}$. 
\end{lemma}

\begin{proof}
    The first claim reduces to the second, and we can assume $\sigma$ is transitive. The ordinary fixed locus is then $(X^n)^{\sigma} = X^{\#O(\sigma)}$. Let $\bar x \to X^n$ be a geometric point at which $\bar M_{X^n, \bar x}$ is nontrivial. We need to show it does not lie in the fixed locus. 
    
    Argue as in Example \ref{ex:counterexamplelogfixedlocus}. The map $\bar x \to X^n \to \Log$ has image a stacky point $B(\Aut(\bar M_{X^n, \bar x}) \ltimes \GG_m^{{\rm rk} \bar M_{X^n, \bar x}})$. 
    
    Write $A$ for the automorphism group $\Aut(\bar M_{X^n, \bar x})$. There is a map 
    \[
        B(A \ltimes \GG_m^{{\rm rk} \bar M_{X^n, \bar x}}) \to B A,
    \]
    and it suffices to assume $\bar x = X^n$ and take the above intersection over $B A$. But then $X^n \times_{BA} X^n = X^n \times A$ and the diagonal and twisted diagonal land in different disjoint copies of $X$: $X \times \{0\}$ and $X \times \{a\}$, where $a \in A$ is a permutation of the factors of $\bar M_{X^n, \bar x} = \bar M_{X, \bar x}^n$. 
\end{proof}


\begin{remark}
    The lemma above implies that the map $X^n\to \af{[X^n/S_n]}$ is not separated. Indeed, the log fixed locus is not a closed subvariety of $X^n$ in general.
\end{remark}

\begin{remark}
    We speculate that a better ``log symmetric product'' exists which is analogous to the log Hilbert scheme of \cite{kennedy-hunt2023loghilbert}. The right construction may very well be to take the stack quotient $[X^n/_{ket} \Sigma_n]$ in the Kummer \'etale topology. The above quotient $[X^n/\Sigma_n]$ is then analogous to the dense open part of the log Hilbert scheme of a s.n.c.\ pair $(X, D)$ on which the points do not meet the boundary divisor. We return to this point in a future paper.     
\end{remark}

Using Theorem \ref{thm:main}, we obtain the following result.

\begin{corollary}
    We have an isomorphism of $k$-vector spaces
    \begin{align*}
        R^*\Gamma(\HHl{[X^n/S_n]}(\OO_{[X^n/S_n]})  &= \left(\bigoplus_{\sigma\in G} R^*\Gamma(\HHl{(X^n)^{\sigma}_{\log}},\OO_{(X^n)^{\sigma}_{\log}})\right)^{S_n}        \\
            &=\bigoplus_{\sigma\in G} R^*\Gamma(\HHl{(X_0)}, \OO_{(X_0)})^{\#S(\sigma)}\otimes R^*\Gamma(\HHl{X},\OO_{X})^{\#F(\sigma)}.
    \end{align*}
\end{corollary}

We remark that $R^*\Gamma(\HHl{(X_0)}, \OO_{(X_0)})$ is simply the (non-log) Hochschild homology of the open $X_0 \subseteq X$. 

The last equality is the consequence of the K\"unneth-formula for log Hochschild homology \cite{leonardi2025logarithmic}.

\subsection{Root stacks}
We begin with a technical lemma.

\begin{lemma}
    Let $X\to \af{X}$ and $Y\to \af{Y}$ be log schemes and a map $\pi:X\to Y$ that induces a map on the Artin fan meaning that we have a commutative diagram
    \[
    \begin{tikzcd}
        X \ar[r, "\pi"]\ar[d] & Y\ar[d]\\
        \af{X}\ar[r] & \af{Y}.
    \end{tikzcd}
    \]
    Then, there exists a pullback map on the log Hochschild homology complexes $\pi^*\HHl{Y}(\OO_Y)\to \HHl{X}(\OO_X)$.
\end{lemma}

\begin{proof}
    We have a commutative diagram
    \[
    \begin{tikzcd}
        X \ar[r, "i_X"]\ar[d, "\pi"] & X\times_{\af{X}}X \ar[d, "(\pi{,}\pi)"]\\
        Y\ar[r, "i_Y"] & Y\times_{\af{Y}}Y
    \end{tikzcd}
    \]
    From the commutativity of the diagram and from base change, we get a sequence of maps
    \[\pi^*i_Y^*i_{Y, *}\OO_Y\to i_X^*(\pi,\pi)^*i_{Y, *}\OO_Y\to i_X^*i_{X,*}\OO_X\]
    that provides the pullback map on the log Hochschild homology complexes.
\end{proof}

Let $(X, D)$ be a log smooth pair: $X$ is a smooth variety, and $D$ a smooth Cartier divisor on $X$. Consider the $n$-th root stack $\sqrt[n]{(X,D)}$ with its logarithmic structure and the projection $\pi:\sqrt[n]{(X,D)}\to X$ to $X$. 

\begin{theorem}\label{thm:rootstack}
    The natural map on log Hochschild complex $\pi^*\HHl{X}(\OO_X)\to \HHl{\sqrt[n]{(X,D)}}(\OO_{\sqrt[n]{(X,D)}})$ is a quasi-isomorphism.
\end{theorem}

\begin{proof}
    It is enough to prove the statement locally. In local coordinates, $X=\Spec A$, and $\sqrt[n]{(X,D)}=[Y/\mu_n]=[\Spec A[t]/(t^n-f)/\mu_n]$ for a generator $f$ of the ideal of the divisor where $\mu_n$ acts by multiplication by $n$-th roots of unities on $t$. The projection map $\pi$ is given by the embedding $A\to A[t]/(t^n-f)$. 
    
    We need to understand the twisted diagonals 
    \[\Spec A[t]/(t^n-f)\to \Spec A[t]/(t^n-f)\times_{\af{\sqrt[n]{(X,D)}}}\Spec A[t]/(t^n-f).\]

    In this case, the Artin fan is $\af{\sqrt[n]{(X,D)}}= [\Aff^1/\GG_m]$. The ambient space can be described locally as
    \[\Spec (A\otimes A)[t, t',u]/(t^n-f\otimes 1, t'^n-1\otimes f, t=ut').\]
    Similarly, the twisted diagonal maps can be described locally as
    \[(A\otimes A)[t, t',u]/(t^n-f\otimes 1, t'^n-1\otimes f, t=ut')\to A\]
    sending $A\otimes A\to A$ via multiplication, $t,t'\mapsto t$, and finally $u\mapsto \mu_n^k$ depending on the twisted diagonal map $\i_{\mu_n^k}$ we consider. This means that the log fixed point locus is empty, except in the case of $g=1$ (the $u$ is mapped to different constants if $g\ne 1$). In other words, the extra components of Theorem \ref{thm:main} are all 0, all contributions come from the self-intersection of the diagonal.

    In the case of the self-intersection of the diagonal
    \[i:\Spec A[t]/(t^n-f)\to \Spec A[t]/(t^n-f)\times_{\af{\sqrt[n]{(X,D)}}}\Spec A[t]/(t^n-f),\]
    the complex $i^*i_*\OO_Y$ is quasi-isomorphic to the symmetric algebra of the shifted conormal bundle $\Sym(N^{\vee}_i[1])$ according to \cite{arinkin2012self}, i.e. 
    \[i^*i_*\OO_Y\cong \Sym(\Omega^{1,\log}_Y[1]).\]
    Again, we use that the diagonal map $i$ splits completely via any of the projection maps $\Spec A[t]/(t^n-f)\times_{\af{\sqrt[n]{(X,D)}}}\Spec A[t]/(t^n-f)\to \Spec A[t]/(t^n-f)$.

    We note that the action of $\mu_n$ is trivial on the generators of $\Omega^{1, \log}_Y$ (which are the pull-back of the generators of $\Omega^{1,\log}_X$), and thus we have a quasi-isomorphism using the log version of HKR (\cite{HABLICSEK2026127})
    \[\pi^*\HHl{X}(\OO_X))\to \HHl{\sqrt[n]{(X,D)}}(\OO_{\sqrt[n]{(X,D)}}))\]
    proving our statement.
\end{proof}

The theorem above tells us that in the case of root stacks there is no contribution from the twisted sections of Theorem \ref{thm:main}, only from the diagonal embeddings.

Furthermore, Theorem \ref{thm:rootstack} implies that log Hochschild homology is invariant under root stack operations.

\begin{corollary}
    Let $(X,D)$ be a log smooth pair with $X$ being proper. Consider the $n$-th root stack $\sqrt[n]{(X,D)}$. Then, we have isomorphisms of vector spaces 
    \[R^k\Gamma(\HHl{X}(\OO_X))\cong R^k\Gamma(\HHl{\sqrt[n]{(X,D)}}, \OO_{\sqrt[n]{(X,D)}})\]
    induced by the pullback map on the log Hochschild complexes.
\end{corollary}

\begin{proof}
    The statement follows from the fact that the derived pushforward of the structure sheaf of the root stack is quasi-isomorphic to the structure sheaf of $X$: $\pi_*\OO_{\sqrt[n]{(X,D)}}=\OO_X$ (see, for instance, \cite{talpo2018infinite}).
\end{proof}

This shows that log Hochschild homology is invariant under root stack operations. This improves the statement of \cite[Theorem C]{HABLICSEK2026127}. We can interpret the above result in two ways. First, the derived category of the root stack has admissible subcategories that the logarithmic derived category should not contain. Therefore, there should be a discrepancy between the Hochschild homology of root stacks \cite{bodzenta2024root,  fu2025hochschild} and the logarithmic Hochschild homology, and this discrepancy comes from the divisor. Second, Hochschild homology is invariant under ``nice" logarithmic alterations, showing that it only depends on the log structure, not the stack representing it. We believe that this statement is important for developing a ``right" version of a derived category of coherent sheaves on logarithmic schemes.

\begin{remark}
    We remark that Theorem \ref{thm:rootstack} holds not just in the case of a log smooth pair $(X,D)$, but also for general simple normal crossing divisors as well. The proof is the same, for the sake of simplicity, we restrict our attention to a smooth divisor.
\end{remark}

\bibliographystyle{alpha}
\bibliography{bib}

\end{document}